\documentclass{article}
\usepackage[T1]{fontenc}
\usepackage{graphicx}
\usepackage[utf8]{inputenc}
\usepackage{booktabs}
\usepackage[ams,todos]{CMImacros}
\usepackage{amsmath}
\usepackage{amsfonts}
\usepackage{amssymb}
\usepackage{float}
\usepackage{array}
\usepackage{booktabs}
\usepackage{amsthm}
\usepackage{enumerate}
\usepackage{cleveref}
\usepackage{orcidlink}
\usepackage{hyperref}
\hypersetup{
    breaklinks=true,
    colorlinks=true,     % <-- use colored text instead of boxes
    linkcolor=blue,      % <-- color for internal links (sections, TOC)
    citecolor=blue,      % <-- color for citations
    urlcolor=blue
}

\usepackage{titlesec}

\titleformat{\subsubsection}
  {\normalfont\scshape}  % small caps, no bold/italic/underline
  {}                     % no label/numbering
  {1em}                  % horizontal space before heading text
  {}                     % no prefix code
  [\vspace{0.5em}]       % space after heading to break the line

\titlespacing*{\subsubsection}
  {0pt}{1.5em}{0.5em}    % {left}{before}{after} spacing

\usepackage[format=plain,justification=justified]{caption}
\usepackage[format=plain,justification=justified]{caption}

\usepackage[backend=bibtex,maxbibnames=99,firstinits=true,style=numeric-comp,url=false,sorting=none]{biblatex}

\DeclareFieldFormat[article]{title}{#1} % print title literally
\DeclareFieldFormat{title}{#1}
\DeclareFieldFormat{journaltitle}{#1}
\DeclareFieldFormat{booktitle}{#1}
\DeclareFieldFormat{pages}{#1}
\DeclareFieldFormat[inproceedings]{title}{#1}

\DeclareNameAlias{author}{given-family}
\renewcommand*{\labelnamepunct}{\addcolon\space}

\renewbibmacro{in:}{}

\AtEveryBibitem{%
  \clearfield{month}%
  \clearfield{isbn}%
  \clearfield{issn}%
  \clearfield{address}%
  \clearfield{note}%
  \clearfield{language}%
  \ifentrytype{article}{}{%
    \clearfield{publisher}%
  }%
}

\renewbibmacro*{title}{%
  \ifentrytype{article}{%
    \printfield{title}%
  }{%
    \printfield[title]{title}%
  }%
}

\renewbibmacro*{title}{%
  \ifentrytype{thesis}{%
    \printfield{title}%
  }{%
    \printfield[title]{title}%
  }%
}

\DeclareFieldFormat{doi}{%
  \adddot\space
  \url{https://doi.org/#1}%
}

\DeclareBibliographyDriver{article}{%
  \usebibmacro{bibindex}%
  \usebibmacro{begentry}%
  \printnames{author}%
  \setunit{\labelnamepunct}\space
  \usebibmacro{title}%
  \adddot\space
  \printfield{journaltitle}%
  \setunit{\space}%
  \mkbibbold{\printfield{volume}}%
  \setunit{\addcomma\space}%
  \printfield{pages}%
  \setunit{\space}%
  \printtext[parens]{\printfield{year}}%
  \newunit\newblock
  \iffieldundef{doi}{}{%
    \printfield{doi}% <--- MODIFIED
  }%
  \usebibmacro{finentry}%
}

\DeclareBibliographyDriver{book}{%
  \usebibmacro{bibindex}%
  \usebibmacro{begentry}%
  \printnames{author}%
  \setunit{\labelnamepunct}\space
  \printfield{title}%
  \adddot\space
  \printlist{publisher}%
  \setunit{\addcomma\space}%
  \printlist{location}%
  \setunit{\space}%
  \printtext[parens]{\printfield{year}}%
  \usebibmacro{finentry}%
}

\DeclareBibliographyDriver{incollection}{%
  \usebibmacro{bibindex}%
  \usebibmacro{begentry}%
  \printnames{author}%
  \setunit{\labelnamepunct}\space
  \printfield{title}%
  \adddot\space
  \printtext{In: }%
  \printnames{editor}%
  \setunit{\addcomma\space}%
  \printtext{(eds.) }%
  \printfield{booktitle}%
  \setunit{\addcomma\space}%
  \printfield{pages}%
  \adddot\space
  \printlist{publisher}%
  \setunit{\addcomma\space}%
  \printlist{location}%
  \setunit{\space}%
  \printtext[parens]{\printfield{year}}%
  \usebibmacro{finentry}%
}

\DeclareBibliographyDriver{misc}{%
  \usebibmacro{bibindex}%
  \usebibmacro{begentry}%
  \printnames{author}%
  \setunit{\labelnamepunct}\space
  \printfield{title}%
  \adddot\space
  \iffieldundef{journal}{%
    \iffieldundef{eprint}{}{%
      \printtext{arXiv:\space\printfield{eprint}}%
    }%
  }{%
    \printfield{journal}%
  }%
  \setunit{\space}%
  \printtext[parens]{\printfield{year}}%
  \usebibmacro{finentry}%
}

\DeclareBibliographyDriver{inproceedings}{%
  \usebibmacro{bibindex}%
  \usebibmacro{begentry}%
  \printnames{author}%
  \setunit{\labelnamepunct}\space% <--- Uses the new definition (now just a space)
  \printfield{title}%
  \adddot
  \setunit{\addspace}\space
  \printtext{In: }%
  \printfield{booktitle}%
  \setunit{\addcomma\space}%
  \printfield{pages}%
  \setunit{\adddot\space} % Use a period before the publisher/year block
  \printlist{publisher}%
  \setunit{\space}%
  \printtext[parens]{\printfield{year}}%
  \usebibmacro{finentry}%
}

\DeclareFieldFormat[phdthesis]{title}{#1}
\DeclareFieldFormat[thesis]{title}{#1}
\DeclareBibliographyAlias{thesis}{phdthesis}

\DeclareBibliographyDriver{phdthesis}{%
  \usebibmacro{bibindex}%
  \usebibmacro{begentry}%
  \printnames{author}%
  \setunit{\labelnamepunct}\space
  \printfield[titlecase]{title}%
  \adddot\space
  \printfield{type}%
  \space
  \printfield{school}%
  \printtext[parens]{\printfield{year}}%
  \usebibmacro{finentry}%
}

\usepackage{authblk}
\usepackage[margin=1.2in]{geometry} 

\newtheorem{theorem}{Theorem}[section]
\newtheorem{lemma}{Lemma}[section]

\theoremstyle{definition}

\theoremstyle{example}

\theoremstyle{corollary}
\newtheorem{corollary}{Corollary}[section]

\theoremstyle{theorem}
\newtheorem{remark}{Remark}[section]

\numberwithin{equation}{section}

\bibliography{string,robo}%

\title{Convergence theory for Hermite approximations under adaptive coordinate transformations}%
% ...
\author[]{Yahya Saleh\thanks{Corresponding Author: yahya.saleh@mathmods.eu}\orcidlink{0000-0002-3235-217X}}
\date{}
\begin{document}
\maketitle

\begin{abstract}\noindent%
Recent work has shown that 
parameterizing and optimizing coordinate
transformations using normalizing flows, \ie, invertible neural networks, can
significantly accelerate the convergence of spectral approximations. We present the first error estimates for approximating functions using
Hermite expansions composed with adaptive coordinate transformations. Our
analysis establishes an equivalence principle: approximating a function $f$ in the
span of the transformed basis is equivalent to approximating the pullback of $f$
in the span of Hermite functions. This allows us to leverage the classical
approximation theory of Hermite expansions to derive error estimates in transformed
coordinates in terms of the regularity of the pullback.
We present an example demonstrating how a nonlinear coordinate transformation
can enhance the convergence of Hermite expansions. Focusing on smooth functions
decaying along the real axis, we construct a monotone transport
map that aligns the decay of the target function with the Hermite basis. This
guarantees spectral convergence rates for the corresponding Hermite expansion. Our analysis provides theoretical insight into the convergence
behavior of adaptive Hermite approximations based on normalizing flows, as
recently explored in the computational quantum physics literature.

\end{abstract}


\tableofcontents

\section{Introduction}

Hermite functions $(h_n)_{n=0}^{\infty}$ form an orthonormal basis for  $L^2$  on the real line, making
them a natural choice for spectral methods on unbounded domains. They are eigenfunctions of the quantum
harmonic oscillator~\cite{Lubich:QCMD} and are routinely used for numerical simulations
in domains varying from quantum molecular physics~\cite{Yurchenko:CSPM2023} to oceanography~\cite{Boyd:DEO2018}.

Recently, several works have incorporated Hermite expansions in conjunction with adaptive
coordinate transformations
 into spectral discretizations of differential equations
arising in molecular quantum physics~\cite{Saleh:JCTC21:5221,Vogt:JCP163:15,
Zhang:JCP161:024103, Li:PRL136:076504} and condensed-matter physics~\cite{Zhang:PRL134:24,
Xie:arXiv2105.08644}.
In particular, for approximating an $L^2$ function $f$, the cited studies replaced the standard ansatz of Hermite
expansions
\begin{equation}
  \label{eq:standard_ansatz}
f \approx \sum_{n=0}^{N-1} \langle f, h_n \rangle_{\mathbb{R}} \ h_n  
\end{equation}
with approximations in the linear span of Hermite functions expressed in nonlinear coordinates parameterized by a transformation \(g\)
\begin{equation}
  \label{eq:flows_ansatz}
  f \approx \sum_{n=0}^{N-1}
  \langle f, h_n\!\circ g \,\sqrt{g'} \rangle_{\Omega}\,
  h_n\!\circ g \,\sqrt{g'}
  \quad \text{or} \quad
  f \approx \sum_{n=0}^{N-1}
  \langle f, h_n\!\circ g \, g' \rangle_{\Omega}\,
  h_n\!\circ g ,
\end{equation}
 where the coordinate transformation
\(g:\Omega \to \mathbb{R}\) is not fixed \emph{a priori} but determined adaptively.

In these approaches, \(g\) was parameterized by an invertible neural network,
referred to as a normalizing flow~\cite{Kobyzev:PAMI43:3964}.
A normalizing flow is a differentiable bijection \(g_\theta\)
with efficiently computable inverse and Jacobian determinant, originally
developed for density estimation in generative machine learning.
In \textcite{Saleh:JCTC21:5221,Vogt:JCP163:15} the parameters \(\theta\) were
optimized by minimizing the Rayleigh quotient of the approximated eigenfunctions
for a specific truncation parameter $N$,
thereby producing  spectral bases with substantially improved
convergence that can be employed with larger $N$.

To illustrate this line of research, consider the task of computing the lowest 23 eigenpairs \((E_m, \Psi_m)_{m=0}^{22}\) of the time-independent Schrödinger equation for the Morse oscillator~\cite{Vogt:JCP163:15}
\begin{equation}
\label{eq:Morse}
  H = -\frac{\hbar^2}{\mu} \frac{\partial^2}{\partial x^2} + D_e \left[1 - \exp(-a_M x)\right]^2,
\end{equation}
where \(\mu\) is the reduced mass, \(D_e\) the dissociation energy, \(a_M\) a
Morse parameter, and \(x = r - r_e\) the displacement from the equilibrium bond
length. An approximation of the form \eqref{eq:flows_ansatz} 
can be used to compute the eigenpairs.

\autoref{fig:1} shows the accuracy of the computed eigenvalues as a function of
the truncation parameter \(N\) for the standard Hermite expansion (ansatz \eqref{eq:standard_ansatz}), adaptive
Hermite expansions (ansatz \eqref{eq:flows_ansatz}) employing a linear
coordinate transformation $g$, and an adaptive Hermite expansion (ansatz
\eqref{eq:flows_ansatz}) employing a nonlinear transformation $g$ parameterized
by a normalizing flow. All coordinate transformations were optimized to minimize
the Rayleigh quotient of the approximated eigenfunctions for $N=23$,
 see \autoref{subsection:Morse} for details. The
results demonstrate that employing a coordinate
transformation substantially improves the accuracy of the computed eigenvalues,
and that the improvement is more pronounced for the nonlinear transformation parameterized by a normalizing flow.

Beyond this illustrative example, the methodology has demonstrated strong
performance in challenging applications, enabling accurate computation of
vibrational spectra for floppy molecules like hydrogen cyanide/hydrogen isocyanide HCN/HNC isomers~\cite{Saleh:JCTC21:5221,
Zhang:JCP161:024103} and examining phase stability in crystalline lithium~\cite{Zhang:PRL134:24}.
\begin{figure}[h]
	  \centering
  \includegraphics[width=0.62\textwidth]{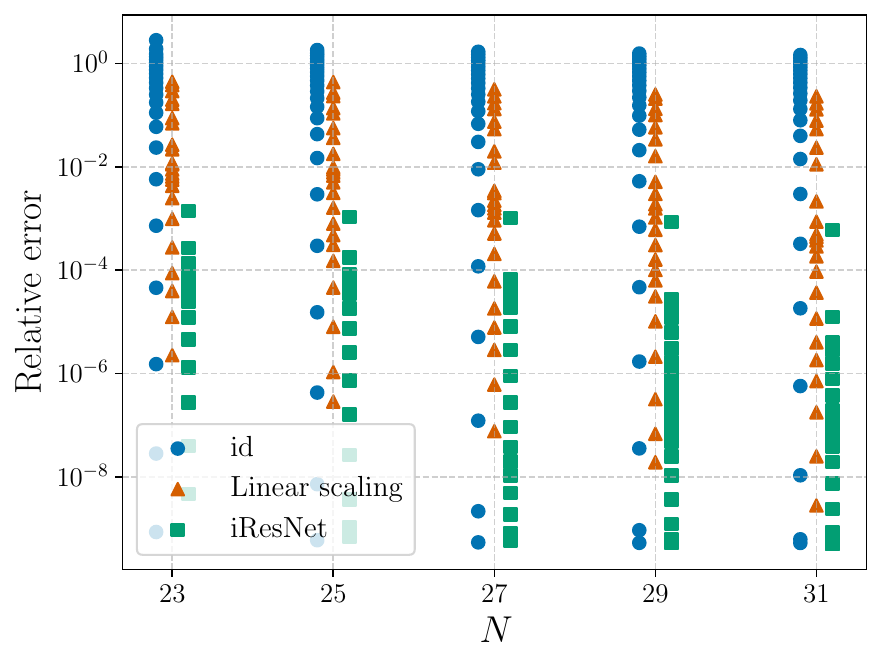}
  \caption{Shown is the relative error $\frac{|\tilde{E}_m - E_m^\text{ref}|}{|E_m^\text{ref}|}$ in the computation of 23
  eigenvalues of the Morse oscillator \eqref{eq:Morse} using Hermite approximations (blue
  circles), Hermite approximations employing a linear coordinate transformation $g$ (red triangles), and Hermite approximations employing a nonlinear coordinate transformation $g$ parameterized by a normalizing flow (green squares) as a
  function of the truncation parameter $N$.}
  \label{fig:1}
\end{figure}

The present work establishes the first rigorous convergence theory for
approximations of the form \eqref{eq:flows_ansatz}.
In particular, we
establish an equivalence principle demonstrating that approximating any $f\in
L^2(\Omega)$ in the linear span of $(h_n \circ g \sqrt{|g'|})_{n=0}^\infty$ is equivalent to
approximating $f \circ g^{-1} \sqrt{|(g^{-1})'|}$ in the linear span of Hermite functions
$(h_n)_{n=0}^\infty$, see \autoref{thm:transfer}. This principle
allows us to leverage the extensive literature on the convergence of Hermite
expansions to derive error estimates for approximations of the form
\eqref{eq:flows_ansatz} in terms of the regularity of $W_{g^{-1}}f$.

Since solutions of various molecular Schrödinger equations are typically smooth
functions that decay along the real axis, we focus on this class of
functions. We show that for any smooth function exhibiting such decay,
there exists a coordinate transformation \(g\) such that the approximation
\eqref{eq:flows_ansatz} achieves spectral convergence rates, see
\autoref{thm:adapt_to_decay} and \autoref{cor:adapt_to_decay}. These transformations are constructed as monotone
transport maps that adapt the decay of the target function at infinity to match
the decay of Hermite functions.

In \autoref{sec:num_exp} we
provide numerical experiments, where optimal coordinate transformations are
learned using a fixed number of basis functions, an approach that recently emerged in computational quantum physics literature~\cite{Saleh:JCTC21:5221,Vogt:JCP163:15}.
We demonstrate that the normalizing flow in
\textcite{Vogt:JCP163:15}, optimized to minimize vibrational energies,
effectively learns a coordinate transformation, such that the pullback of the target eigenfunction exhibits increased decay at infinity, thereby accelerating the convergence of the approximation.

A survey of application of Hermite expansions in conjunction with coordinate
transformations is in order.
Employing coordinate transformations in spectral methods is a well-established
technique in numerical analysis~\cite{Boyd:SM2000, Canuto:SpectralMethods,
Guo:SM:1998, Shen:SM:2024}. This technique may be beneficial for a variety of
reasons, such as to map the computational domain to the domain of the basis
functions~\cite{Shen:CCP5:195}, or to resolve singularities and internal boundary
layers~\cite{Boyd:SM2000}. In conjunction with Hermite expansions, the most
common choice of coordinate transformation is a scaling
factor~\cite{Tang:SiamSC14:594, Boyd:MC35:152,Boyd:JCP54:382}. Various rationale
have been proposed for the choice of optimal scaling factors. For example,
\textcite{Tang:SiamSC14:594} proposed to choose the scaling factor as to balance
the maximum root of the Hermite function of order $N-1$ and the length of the
finite support of the target function. 
Recent studies demonstrated that
optimal scaling factor should balance the spatial and frequency truncation
errors~\cite{Hu:SiamJNA64:125, Jaming:arXiv1407:1293}.
In application to time-dependent problems, the scaling factor can be chosen
adaptively in time to account for the changing spatial extent of the solution~\cite{Xia:SiamJSC43:3244}. 

Despite the widespread use of coordinate transformations in Hermite-based
spectral methods, most studies have focused on the analysis of scaling factors,
while studies of more general nonlinear coordinate transformations considered
only \emph{a priori} defined maps~\cite{Shen:CCP5:195}. A convergence theory
for Hermite approximations employing adaptive coordinate transformations, such as those parameterized
by normalizing flows, has been lacking. The present work aims to fill this gap.
\section{Background}
\label{sec:preliminaries} 
 Hermite functions
$(h_n)_{n=0}^\infty$ are defined by
$$	  h_n(x) := \frac{1}{\sqrt{2^n n! \sqrt{\pi}}} H_n(x) \exp({-x^2/2}),$$
	where $H_n$ is the $n$-th Hermite polynomial. These functions arise naturally as the eigenfunctions of the quantum harmonic oscillator, whose Hamiltonian is given by
\begin{equation}
\label{Eq:HO}
H_{\text{HO}} = -\frac{1}{2} \frac{d^2}{dx^2} + \frac{1}{2} x^2.
\end{equation}
  
  Like any basis, Hermite functions can be employed in conjunction with a differentiable change of
coordinates $g: \Omega \to \mathbb{R}$, where $\Omega \subseteq \mathbb{R}$
is an open unbounded set denoting the computational domain. 
  
  When defining a differentiable change of coordinates,
  one needs to make sure that the transformed basis functions
  \begin{equation}
    \label{eq:induced_Riesz}
    (\phi_n)_{n=0}^\infty := (C_g h_n )_{n=0}^\infty \qquad \text{where} \qquad C_g f = f \circ g,
    \end{equation}
     form a complete basis in $L^2(\Omega)$.
          A first requirement is that
     the transformed functions lie in $L^2(\Omega)$, which can be guaranteed if 
          $g'$ is bounded almost everywhere. This condition ensures that the
          operator $C_g$ maps $L^2(\mathbb{R})$ continuously into $L^2(\Omega)$~\cite{Takagi:FS232:321}.   
    Ensuring completeness of the transformed basis is crucial for guaranteeing the convergence of approximations constructed from a finite subset of basis functions~\cite{Saleh:PAMM23:e202200239}.
    Recently, necessary and sufficient conditions to this end were characterized
    and expansion coefficients were derived~\cite{ Saleh:CAOT20:21}. We
    summarize these results in the following theorem.
    \begin{theorem}
      \label{thm:riesz}
      Let $g: \Omega \to \mathbb{R}$ be a differentiable map and assume that $g'$
      is bounded almost everywhere. The sequence
      $(\phi_n)_{n=0}^\infty$ defined in \eqref{eq:induced_Riesz} is a Riesz
      basis of
      $L^2(\Omega)$ if and only if $g$ is injective and $g'$ is bounded away from zero almost everywhere. In this case, any
      function $f \in L^2(\Omega)$ can be represented by the unique series
      expansion
      \begin{equation}
        \label{eq:riesz_expansion}
        f = \sum_{n=0}^\infty \langle f, \tilde{\phi}_n \rangle_{\Omega} \ \phi_n,
      \end{equation}
      where $\langle \cdot, \cdot \rangle_{\Omega}$ denotes the inner product in
      $L^2(\Omega)$ and $(\tilde{\phi}_n)_{n=0}^\infty:= (\phi_n \ g')_{n=0}^\infty$ is the dual basis of $(\phi_n)_{n=0}^\infty$.
    \end{theorem}
    \begin{proof}
      The result follows from \textcite[Corollary 4.1]{Saleh:CAOT20:21}. 
    \end{proof}
  
    While a bi-Lipschitz map $g$ ensures that the sequence
    \eqref{eq:induced_Riesz} forms a Riesz basis, maintaining the orthonormality
    of the basis functions  $(\phi_n)_{n=0}^\infty$  is only possible under the
    restrictive condition that  $g^{\prime} = 1$  almost
    everywhere~\cite{Saleh:JCTC21:5221}. However, it is possible to account
    for the volume distortion introduced by the change of coordinates and adjust accordingly. Let $g: \Omega \to \mathbb{R}$ be a $C^1$ diffeomorphism and consider the sequence 
    \begin{equation}
      \label{eq:induced_ons}
      (\chi_n)_{n=0}^\infty := (W_g h_n)_{n=0}^\infty \qquad \text{where} \quad (W_g f)(x)  := (f \circ
  g)(x) \ \sqrt{|g'(x)|}
      \end{equation}
       Here, $W_g$ is a weighted composition operator induced by $g$. The
       operator maps $L^2(\mathbb{R})$ to $L^2(\Omega)$. The sequence \eqref{eq:induced_ons} is
  an orthonormal basis in $L^2(\Omega)$, \ie, any $f$ in $L^2(\Omega)$ can be represented by the unique
  series expansion
  $$
  f = \sum_{n=0}^\infty \langle f, \chi_n \rangle_{\Omega} \ \chi_n.
  $$

  Throughout this work, we assume that bases of the form
  \eqref{eq:induced_ons} are induced by maps $g$ with  \(g' > 0\) almost
  everywhere.

  In recent literature, various normalizing-flow architectures have been
  proposed to model the coordinate transformation \(g\) such as
  invertible residual neural networks (iResNets)~\cite{Saleh:JCTC21:5221,
  Vogt:JCP163:15}, and RealNVP~\cite{
  Zhang:JCP161:024103,Zhang:PRL134:24,}. For our numerical experiments, we
  employ the iResNet architecture~\cite{Behrmann:ICML2019:573}, which is a neural network composed of
  residual blocks of the form
\begin{equation}
\label{eq:iResNet_block}
x \mapsto x + \sigma(W_2 \sigma(W_1 x + b_1) + b_2),
\end{equation}
where \(W_1, W_2\) are weight matrices, \(b_1, b_2\) are bias vectors, and
\(\sigma\) is a nonlinear activation function with a Lipschitz constant less than
one. Invertibility is guaranteed by
constraining the spectral norms of the weight matrices to be less than one. The
inverse can be computed efficiently by fixed-point iterations.

  The error estimates for the Riesz and orthonormal bases of the forms
\eqref{eq:induced_Riesz} and \eqref{eq:induced_ons}, respectively, differ by
norm equivalence constants arising from the lack of orthonormality. For
simplicity, we present the equivalence principle for the
orthonormal bases \eqref{eq:induced_ons} in \autoref{sec:convergence_iHermite}.
The equivalence principle for
the Riesz bases is deferred to \autoref{app:error_comp}.

  \section{Convergence theory of Hermite expansions}
  \label{sec:error_estimates_Hermite}

  The convergence of Hermite expansions for approximating various target functions
has been extensively studied~\cite{Hille:DMJ5:4, Hille:TAMS47:1, Boyd:SM2000, Canuto:SpectralMethods,
Guo:SM:1998, Shen:SM:2024, Lubich:QCMD}.
The convergence of Hermite expansions relies on the smoothness of the
target function and on the rate at which the function decays at
infinity. Early work of \textcite{Hille:DMJ5:4, Hille:TAMS47:1} studied Hermite
approximations for functions that are analytic in a strip around the real-axis. This
work has been extended by \textcite{Boyd:MC35:152, Boyd:JCP54:382}, who performed asymptotic
analysis of the expansion coefficients pertaining to analytic and entire
functions. In recent work, insights on the convergence of Hermite expansions for
analytic functions have been formalized in an \emph{a priori} error upper bound ~\cite{Wang:arXiv2312:07940}. A functional analytic approach to the convergence of Hermite
expansions of less regular functions is presented in \textcite{Shen:SM:2024}, where the
smoothness and decay of the target function are measured by the powers of the
Dirac ladder operator $A$. Recently, a new framework was introduced, where the
smoothness and decay of the target function are measured by spatial and
frequency truncation errors~\cite{Hu:arXiv2412:08044}.

 We briefly review the functional analytic approach in what follows.

  An important object for the analysis of Hermite expansions is the Dirac
  ladder operator $A$, also known as the annihilation operator, which is defined as
  $$
  A: D(A) \to L^2(\mathbb{R}), \quad A: f \mapsto \frac{1}{\sqrt{2}}(xf + \frac{d}{dx} f).
  $$
  The operator $A$ has the property that $A h_0 = 0$ and $A h_n =
  \sqrt{n} h_{n-1}$ for any $n \geq 1$.  

Denote by $P_N$ the projection operator on the linear span of $N$
Hermite functions, \ie, $$P_N: L^2(\mathbb{R}) \to \operatorname{span}(h_0, \ldots,
h_{N-1}), \quad P_N : f \mapsto \sum_{n=0}^{N-1} \langle f, h_n \rangle_{\mathbb{R}} \ h_n,$$
where $\langle \cdot, \cdot \rangle_{\mathbb{R}}$ denotes the inner product in
$L^2(\mathbb{R})$.

\begin{theorem}
\label{thm:error_estimate}
  For any positive integers $l \leq N$ and any $f \in D(A^l)$, 

	\begin{equation}
		\label{eq:error_estimate}
		\|f - P_{N} f\|_{L^2(\mathbb{R})} \leq c \ \frac{1}{\sqrt{N(N-1)(N-2)\dots (N-l+1)}} \ \|A^l f\|_{L^2(\mathbb{R})},
	\end{equation}
  for some $c>0$ independent of $f, N$ and $l$.
\end{theorem}
\begin{proof}
See~\textcite[Corollary 7.3]{Shen:SM:2024} or ~\textcite[Theorem III.1.2]{Lubich:QCMD}.
\end{proof}
We note that the error estimate \eqref{eq:error_estimate} does not naturally
extend to fractional Sobolev spaces, since the operator $A$ is not self-adjoint. 

In view of the definition of $A$, the estimate \eqref{eq:error_estimate} implies that the convergence rate of the approximation depends on the
smoothness of the target function and on the rate at which it decays at
infinity. For functions
decaying algebraically as $|x| \to \infty$,  the convergence rate is
$\mathcal{O}(N^{-l/2})$. For Schwartz functions, \ie, for smooth functions whose
derivatives of any order decay faster than any polynomial, the
convergence rate is faster than any finite power of $N$ and the so-called
spectral convergence is attained. 

When spectral convergence is attained, the expansion coefficients
typically decay as $e^{-\nu n^\kappa}$  for some  $\nu, \kappa > 0$.
However, the value of $\kappa$ is critical—whether $\kappa = 1$ or $\kappa =
0.001$ leads to vastly different convergence behaviors.
\autoref{thm:error_estimate} does not provide insight into the value of
$\kappa$.

Asymptotic analysis performed by \textcite{Hille:DMJ5:4, Hille:TAMS47:1,
Boyd:SM2000} shows that geometric convergence, \ie, $\kappa=1$, is attained for
entire functions decaying as $\exp(-\alpha |x|^2)$ on the real-axis. Decrease or increase in the decay rate
of the target function reduces the rate of convergence. The following theorem summarizes these results.

\begin{theorem}[Approximating entire functions exhibiting exponential decay]
  \label{thm:asymptotic_entire_gaussian}
  Let $f \in L^2(\mathbb{R})$ admit an entire extension to the complex plane and assume that $f$ satisfies the following decay condition 
  \begin{equation}
    \label{eq:assumption_entire_gaussian}
    |f(x)| \leq c \exp(-\alpha |x|^\gamma) \qquad \text{for all } x \in \mathbb{R},
  \end{equation}
  where $c, \alpha>0$ and $\gamma \geq 1$. Then, the expansion coefficients of
  $f$ in the Hermite basis satisfy
  \begin{equation}
    \label{eq:entire_gaussian_estimate}
    |\langle f, h_n \rangle_{\mathbb{R}}| \sim c' \exp(-\nu n^{\kappa}),
  \end{equation}
  where $c', \nu >0$ and
\[
\kappa =
\begin{cases}
\dfrac{\gamma}{2}, & \text{if } 1 \le \gamma < 2, \\
\dfrac{\gamma}{2(\gamma-1)}, & \text{if } \gamma > 2,  \\
1, & \text{if } \gamma = 2.
\end{cases}
\]
\end{theorem}
\begin{proof}
  The proof relies 
  on asymptotic analysis of the expansion coefficients of $f$ in the Hermite basis,
  see \textcite[Chapter 17, Theorem 34.]{Boyd:SM2000}.
\end{proof}

\section{Convergence theory of Hermite expansions under adaptive coordinate transformations}
\label{sec:convergence_iHermite}

We demonstrate in this section that approximating an unknown function \(f\) in
the linear span of the adaptive orthonormal basis \eqref{eq:induced_ons} is
equivalent (\emph{via} a unitary transformation) to approximating the pullback function
\[u = W_{g^{-1}}f
\] in the linear span of the standard Hermite functions. This equivalence
principle allows us to transfer error estimates for approximating functions
in the linear span of Hermite functions, such as those reviewed in
\autoref{sec:error_estimates_Hermite}, to error estimates for approximations
in the linear span of adaptive Hermite functions. 

\subsection{An equivalence principle for Hermite approximations employing adaptive coordinate transformations}

The following lemma summarizes useful properties of the operator \(W_g\) that
are used in the proof of the equivalence principle.
\begin{lemma}
  \label{lem:props_Wg}
  Let $g: \Omega \to \mathbb{R}$ be a $C^1$ diffeomorphism. Then, $W_g$ is unitary
  and its inverse $W_g^{-1}$ is a weighted composition operator induced by the inverse map $g^{-1}$, \ie, $W_g^{-1} = W_{g^{-1}}$.
\end{lemma}
\begin{proof}
  The claims can be readily derived from the definition of $W_g$ and the fact that $g$ is a diffeomorphism.
\end{proof}
 Let 
\begin{equation*}
  P_{N, W}: L^2(\Omega) \to \operatorname{span}(W_g h_0, \dots, W_gh_{N-1}), \quad  P_{N, W}: f \mapsto \sum_{n=0}^{N-1} \langle f, W_g h_n \rangle_\Omega \ W_g h_n,
\end{equation*}
denote the projection operator on the linear span of the adaptive orthonormal basis
\eqref{eq:induced_ons}. 

\begin{theorem}[Equivalence principle for orthonormal bases]
\label{thm:transfer}
For any $f \in L^2(\Omega)$, 
\[
\|f-P_{N,W}f\|_{L^2(\Omega)}
=
\|W_{g^{-1}}f - P_N(W_{g^{-1}}f)\|_{L^2(\mathbb{R})}.
\]
\end{theorem}
\begin{proof}
  	By \autoref{lem:props_Wg}, $W_g:L^2(\mathbb{R})\to L^2(\Omega)$ is a unitary operator and $W_{g^{-1}} = W_g^{-1}$, \ie,
	$W_{g^{-1}}$ is the inverse of $W_g$. Thus,
   $$\|f - P_{N,W}f\|_{L^2(\Omega)}=\|W_{g^{-1}}
	(f-P_{N,W}f)\|_{L^2(\mathbb{R})}.$$
   The proof is completed by noting that
	\begin{align*}
	W_{g^{-1}} P_{N,W} f &= \sum_{n=0}^{N-1} \langle f, W_g h_n \rangle_{\Omega} \ h_n \\ &= \sum_{n=0}^{N-1} \langle W_{g^{-1}} f,  h_n \rangle_{\mathbb{R}} h_n.
	\end{align*}
\end{proof}
\autoref{thm:transfer} also demonstrates that the expansion coefficients of $f$ in the adaptive orthonormal basis \eqref{eq:induced_ons} are equal to the expansion coefficients of $W_{g^{-1}}f$ in the standard Hermite basis. 

\subsection{Spectral convergence \emph{via} decay matching}

We provide an example demonstrating how a nonlinear coordinate transformation can enhance the convergence of Hermite expansions. 
Specifically, we construct a map \(g\) such that the pullback function \(u =
W_{g^{-1}}f\) has Gaussian decay at infinity. 
This result shows that, compared with classical Hermite expansions, adaptive
Hermite approximations can achieve spectral convergence for a larger class of
functions. 

The construction of such coordinate transformations relies on the theory of optimal transport, which provides a systematic way to construct maps that push one probability measure to another. In our case, we construct a map that pushes the probability measure associated with the squared target function \(f^2\) to the Gaussian measure associated with the squared Hermite functions.

We adopt the following terminology. A function \( f \)
is said to have an algebraic decay along the real axis if there exist constants \( c, \alpha > 0 \)
such that
\[
|f(x)| \leq \frac{c}{(1+|x|)^\alpha}, \quad \text{for all } x \in \mathbb{R}.
\]
We say that \( f \) has exponential decay along the real axis if there exist constants \( c, \alpha, \gamma > 0 \) such that
\[
|f(x)| \leq c \exp\!\big(-\alpha |x|^\gamma\big), \quad \text{for all } x \in \mathbb{R}.
\]
In this case, the decay is classified as follows: it is sub-Gaussian if \( \gamma < 2 \), Gaussian if \( \gamma = 2 \), and super-Gaussian if \( \gamma > 2 \).

\begin{theorem}
  \label{thm:adapt_to_decay}
  Assume $f: \Omega \to \mathbb{R}$, $f\in L^2(\Omega)$ has a real axis decay
  that is at most exponential. Then, there exists a diffeomorphism $g: \Omega
  \to
  \mathbb{R}$ such
  that $W_{g^{-1}}f$ has Gaussian decay along the real axis.
\end{theorem}
\begin{proof}

Let $f$ be an arbitrary function satisfying the hypothesis of the theorem. Define $p: \mathbb{R} \to \mathbb{R}^+$ 
by
\begin{equation}
\label{eq:definition_p}
  p(x) := \frac{f^2(x) + \exp(-\exp(x^2))}{\int_{\mathbb{R}} f^2(y) + \exp(-\exp(y^2)) \ dy}.  
\end{equation}

$p$ is a probability density function that is non-zero everywhere. Moreover, $p$
has the same decay behavior as $f$ since $\exp(-\exp(x^2))$ decays faster than
any function exhibiting exponential decay with finite exponent $\gamma$.

  The following argument is an adaptation of a standard result in optimal transport theory; see, \eg, \textcite{Villani:OT:2003}. The argument is provided here for the sake of completeness. 

Set $k(x) = \frac{1}{\sqrt{ \pi}}\exp(-x^2)$.
Note that $p$ and $k$ represent the densities of two probability measures $\mu$
and $\nu$ on $\mathbb{R}$, respectively, \ie, $d \mu(x) = p(x) \ dx$ and $d
\nu(x) = k(x) \ dx$. Define their cumulative distributions by
$$
P(x) = \int_{-\infty}^x p(t) \ dt, \quad K(x) = \int_{-\infty}^x k(t) \ dt.
$$
Note that $P$ and $K$ are strictly monotone functions and therefore invertible. Set $g = K^{-1} \circ P$. 
Clearly, $g$ is invertible and it pushes $\mu$ to $\nu$, \ie, for any
measurable $A \subseteq \mathbb{R}$ we have
$$
\int_{\mathbb{R}} \chi_A \bigl( g^{-1} (x)\bigr) \ k(x) \ dx = \int_{\mathbb{R}} \chi_A (x)\ p(x) \ dx,
$$
where $\chi_A$ denotes the characteristic function of the set $A$. It follows
that 
$$
\int_{A} \frac{1}{\sqrt{ \pi}}\exp\bigl(-g(x)^2\bigr) \frac{1}{|(g^{-1})' \circ g(x)|} \ dx = \int_{A} p(x) \ dx.
$$
Since both $P$ and $K$ are Lipschitz continuous, smooth and strictly monotone, so is
$g$.
Note that $\bigl((g^{-1})'\circ g\bigr)^{-1}= g'$. This implies that for almost every $x \in \mathbb{R}$ we have 
$$
\frac{1}{\sqrt{ \pi}}\exp(-x^2)  = (p \circ g^{-1})(x) |(g^{-1})'(x)|.
$$

We deduce that $$f^2 \circ g^{-1} \ |(g^{-1})'| = \frac{c}{\sqrt{ \pi}}\exp(-x^2) - \exp(-\exp(g^{-1}(x)^2))\ |(g^{-1})'| \leq \frac{c}{\sqrt{ \pi}}\exp(-x^2),$$

where $c = \int_{\mathbb{R}} f^2(y) + \exp(-\exp(y^2)) \ dy$.
Therefore, $W_{g^{-1}}f$ decays as $\exp(-x^2/2)$.

Note that adding $\exp(-\exp(x^2))$ in the definition of $p$ guarantees that
$g$ is 
strictly monotone, even if $f^2$ is not strictly positive everywhere.
\end{proof}
\begin{corollary}
  \label{cor:adapt_to_decay}
  Let $f \in L^2(\Omega)$ be a smooth function
  satisfying the hypothesis of \autoref{thm:adapt_to_decay}. Then, there
  exists a coordinate transformation $g$ and an orthonormal basis $(W_g
  h_n)_{n=0}^\infty$ such that the approximation \eqref{eq:flows_ansatz}
  converges spectrally fast to $f$. If in addition $W_{g^{-1}}f$ admits an entire
  extension to the complex plane, then the approximation \eqref{eq:flows_ansatz}
  converges geometrically fast to $f$.
\end{corollary}
\begin{proof}
The diffeomorphism $g=K^{-1} \circ P$ constructed in the proof of
\autoref{thm:adapt_to_decay} ensures that the pullback function $W_{g^{-1}}f$
decays as $\exp(-x^2/2)$. Together with the smoothness of $f$, this implies that
$W_{g^{-1}}f$ is a Schwartz function and hence belongs to $ D(A^l)$ for any $l \in \mathbb{N}$. Therefore, the approximation
\eqref{eq:flows_ansatz} converges spectrally to $f$ by
\autoref{thm:error_estimate}.

If, in addition, $W_{g^{-1}}f$ admits an entire extension to the complex plane,
then the convergence of the approximation \eqref{eq:flows_ansatz} is geometric
by \autoref{thm:asymptotic_entire_gaussian}.
\end{proof}

\autoref{cor:adapt_to_decay} extends the class of functions for which spectral
convergence can be attained by Hermite expansions. In particular, spectral
convergence can be attained for functions with algebraic decay along the real
axis, which are not well approximated by classical Hermite expansions. Geometric
convergence can
also be attained for functions with sub- or super-Gaussian decay along the real
axis, as long as the pull back function $W_{g^{-1}}f$ admits an entire extension to the complex plane.

We observe that one can also improve the convergence of entire functions
with real-axis decay by
increasing or decreasing the decay rate of the pullback function $W_{g^{-1}}f$
at infinity. This can be achieved using, \eg, a power-law coordinate transformation $g$, such as 
\begin{equation}
\label{eq:example_g}
g: \mathbb{R} \to \mathbb{R}, \quad g(x) = \frac{x}{(x^2 + \epsilon)^{\beta}}, \quad \text{with } \epsilon > 0, \ \beta \in \mathbb{R},\  \beta<0.5.
\end{equation}

\( g \) is smooth, and strictly increasing. For $\beta >0$, $|g(x)| < \
|x|^{1-2\beta}$ for any $x \in \mathbb{R}$. 
For functions decaying slower than Hermite functions, the decay rate of the pullback function $W_{g^{-1}}f$ can be
increased by choosing $\beta$ such that $0< \beta<1/2$. In such a case, we have
a sublinear coordinate transformation that can be used to increase the decay
rate of a target function. Similarly, for functions decaying faster than Hermite
functions, the decay rate of the pullback function $W_{g^{-1}}f$ can be
decreased by choosing $\beta$ such that $\beta < 0$. In such a case, we have a
superlinear coordinate transformation, \ie, $g$ satisfies $|g(x)| > \ |x|^{1-2\beta}$ for any $x \in
\mathbb{R}$. Such a transformation can be used to decrease the decay
rate of a target function.

\section{Numerical experiments}
\label{sec:num_exp}
\subsection{Interpolation}
In the following, we study the approximation of the target functions
\begin{align}
  \label{eq:target_function}
  f_1(x) &= c_1 \ \sin(x + 1.2)\, \frac{1}{(1+x^2)^4}, \\
  f_2(x)&= c_2 \ \sin(x + 1.2)\, \exp(-x^4/2 + x^2/2),  \nonumber 
\end{align}
for $x \in \mathbb{R}$. We used 10000 uniformly spaced sample points to sample
$(-31,31)$ and $(-13,13)$ for $f_1$ and $f_2$, respectively. We
denote by $p_1$ and $p_2$ the decay factors of $f_1$ and $f_2$, respectively,
\ie, $p_1(x) = 1/(1+x^2)^4$ and $p_2(x) = \exp(-(x^4-x^2)/2)$. We set
$c_1$ and $c_2$ such that $\|p_1\|_{L^2(\mathbb{R})} = \|p_2\|_{L^2(\mathbb{R})}
= 1$. 

For $f_2$, standard Hermite approximations exhibit suboptimal convergence rates due to the
mismatch between the decay behavior of $p_2$ and the intrinsic
Gaussian decay $\exp(-x^2/2)$ of Hermite functions. Asymptotic analysis of Boyd shows that the error decays as $\mathcal{O}(s \exp(-q N^{2/3}))$ for some constants $q,
s > 0$~\cite{Boyd:MC35:152}. For $f_1$, the Hermite expansion is suboptimal
due to decay mismatch and the presence of singularities at $z = \pm i$. Asymptotic
analysis by Boyd demonstrates that the Hermite coefficients $a_n$ decay as
$\mathcal{O}\bigl(\frac{1}{N^{41/12}}\bigr)$, see \textcite[Eq. 17.28]{Boyd:SM2000}

We constructed approximations of the target functions \( f_1, f_2 \) in the linear
spans of four different bases. For $i=1,2$, the approximations take the following forms
\begin{align*}
\tilde{f}^N_i(x) &= \sum_{n=0}^{N-1} c_{n,i} h_n(x), \\
\tilde{f}^{N,g}_i(x) &= \sum_{n=0}^{N-1} c_{n,i}^{g} \, W_{g_{i}} h_n(x).
\end{align*}

We modelled the coordinate transformations $g_i$ using three different
approaches: a linear model composed of a scaling and a shift, a power-law model of the form given in
\autoref{eq:example_g}, and an iResNet architecture.

We optimized the parameters of all three models as to align the decaying functions $p_1$ and $p_2$ to the
Gaussian decay. This was achieved by solving

\begin{equation}
  \label{Eq:opt_problem}
\mathcal{L(\theta)}:= \zeta \ \mathcal{L}_1(\theta) + (1-\zeta) \ \mathcal{L}_2(\theta) \longrightarrow \min_{\theta},
\end{equation}
where 
\begin{equation*}
\mathcal{L}^2_1(\theta) =  \ \sum_{k<10000} \bigl( \frac{1}{\sqrt{\pi}} \exp(-g_\theta(x_k)^2) g_\theta'(x_k)- p_i (x_k)\bigr)^2 \ w_k, \end{equation*}
and
\begin{equation*}
\mathcal{L}_2^2(\theta) =  \sum_{k<10000} \bigl( \frac{1}{\sqrt{\pi}} \exp(-g_\theta(x_k)^2) g_\theta'(x_k)- p_i(x_k)  \ \bigr)^2 (\frac{1+x_k^2}{p_i(x_k)+10^{-30}}) \ w_k
\end{equation*}
Here, $x_k$ are the quadrature points and $w_k$ are the corresponding quadrature
weights. The loss $\mathcal{L}_1$ is the discretized $L^2$ distance between the
Gaussian density and the pullback density induced by \( g_\theta \).
$\mathcal{L}_2$ is a weighted version of $\mathcal{L}_1$ that punishes the
mismatch between the Gaussian density and the pullback density more heavily
in the tail regions. This encourages the coordinate transformation to better align the decay of the target function with that of the Hermite functions, which is crucial for achieving faster convergence.

The parameter $\zeta \in [0,1]$ controls the relative importance of
these two objectives. We set $\zeta = 0.9$ for $f_1$ and $\zeta = 0.3$ for $f_2$.

For approximations of $f_2$ employing the power-law transformation
\eqref{eq:example_g} we fixed the value of $\epsilon$ to $2.45$ and optimized
the value of $\beta$. Whereas for $f_1$, 
we optimized the values of $\epsilon$ and $\beta$ jointly. 
For the iResNet model, we used a 10 blocks, each composed of two hidden layers
with 8 units and one output unit. We employed the lipswish activation function.
We optimized the parameters of all the models using the Adam optimizer. More details on the computational setup are provided in \autoref{app:error_comp}.

For \( f_2 \), we also employed a coordinate transformation as defined in \autoref{eq:example_g}, where the parameters \( \epsilon \) and
\( \beta \) were fixed to $8$ and $-0.45$, respectively, based on asymptotic considerations (see discussion below).

Once the parameters of all transformations were determined, we computed the
coefficients $c_{n,i}^g, c_{n,i}$ by projecting the target
functions on the corresponding basis functions for various values of $N$. We then computed the approximation errors $\|f_i - \tilde{f}_i\|_{L^2(\mathbb{R})}$ for each basis and plotted the results in \autoref{fig:2}.

\begin{figure}[h]
	\centering
  \includegraphics[width=\textwidth]{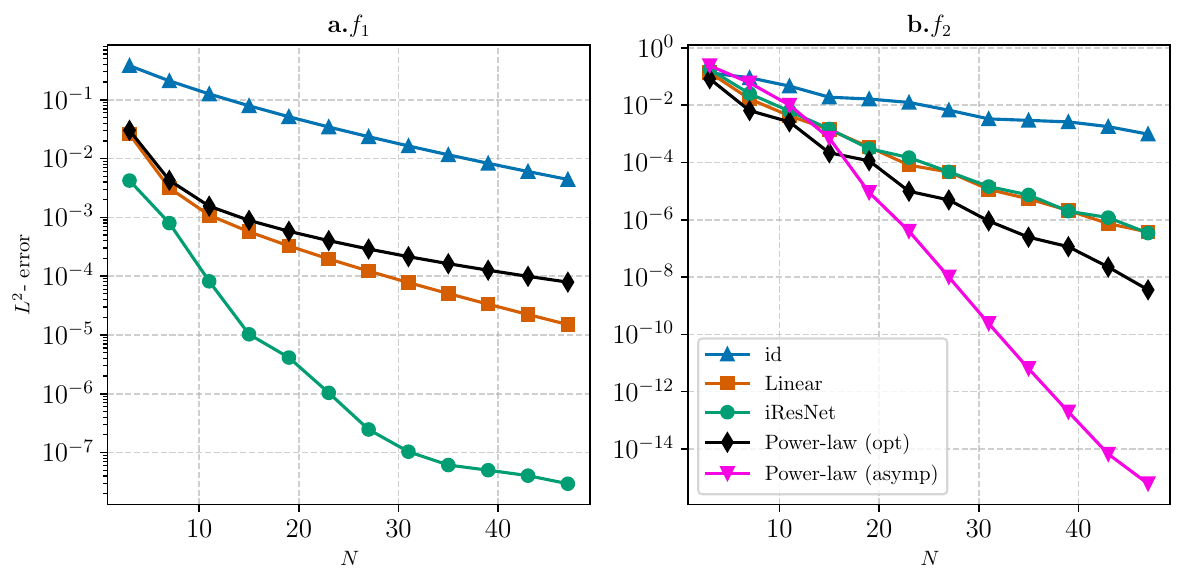}
  \caption{Empirical convergence of function approximations using Hermite functions (blue upward triangles), Hermite functions with a linear coordinate transformation (red squares), Hermite functions with a power-law transformation (black diamonds), Hermite functions with a power-law transformation using asymptotically selected parameters (purple upward triangles), and Hermite functions with an iResNet-based transformation (green circles). Results are shown for the target functions \eqref{eq:target_function}.}
  \label{fig:2}
\end{figure}

As expected, bases tailored to the decay behavior of the target function yield
faster convergence and greater accuracy than standard Hermite approximations.

To gain a more concrete insight into the improvement of convergence enabled by
the use of adaptive coordinate transformations, we estimated the convergence of
the approximations as follows.
For $f_1$, approximations are expected to converge at a rate $N^{-l}$ for some
$l > 0$, see \autoref{thm:error_estimate}. To estimate $l$, we performed linear regression on the
logarithm of the approximation errors as a function of $\log(N)$ for each basis. For $f_2$, approximations are expected to converge at a rate $
\exp(-\nu N^{\kappa})$ for some constants $\nu, \kappa > 0$. To
estimate $\nu$ and $\kappa$, we performed linear regression on the logarithm of the negative logarithm of the
approximation errors as a function of $\log(N)$ for each basis. Results are
provided in \autoref{tab:2}.

\begin{table}[h]
\centering
\begin{tabular}{|c|c|c|}
\hline
\textbf{Transformation} & \textbf{$f_1$} & \textbf{$f_2$} \\
\hline
Hermite & $l = 1.67$ & $\nu=1.06,\ \kappa = 0.47$ \\
\hline
Linear & $l = 2.64$ & $\nu=0.59,\ \kappa = 0.72$ \\
\hline
iResNet & $l = 4.87$ & $\nu=0.79,\ \kappa = 0.77$ \\
\hline
Power-law (opt) & $l = 2.10$ & $\nu=1.12,\ \kappa = 0.73$ \\
\hline
Power-law (asymp) & -- & $\nu=0.28,\ \kappa = 1.25$ \\
\hline
\end{tabular}
\caption{Estimated convergence rates for different coordinate transformations applied to the target functions \eqref{eq:target_function}.}
\label{tab:2}
\end{table}

For $f_1$, the optimization of linear and power-law coordinate transformation
improves the convergence rate of Hermite approximation by a factor of approximately
$1.58$ and $1.2$, respectively. 
Approximations employing an iResNet transformation
of coordinates improve the convergence rate of Hermite approximations by a factor of
$\approx 2.92$. The optimal parameters of the iResNet, linear, and power-law
models achieve a total loss of $\mathcal{L}(\theta)\approx 10^{-3}$, $2.7 \times 10^{-2}$,
and $3 \times 10^{-2}$, respectively. The better performance of the iResNet model can be attributed to its greater expressivity compared to the linear and power-law models, which allows it to better align the decay of the target function with that of the Hermite functions. 

Since the convergence of approximations of $f_1$ are also affected by the
presence of singularities, it is instructive to look at the singularities of the
pullback for the three coordinate systems. For all three coordinate systems, the
singularities of the pullback function are located at $g(\pm i)$. This quantity
can be computed exactly for the linear and power-law models and can be estimated
using Taylor expansion for the iResNet model. We calculated $g(\pm i)$ for the
three coordinate systems and found that the distance to the real axis of the
closest singularity is approximately given by $2.67$ for the linear model,
$2.58$ for the power-law model, and $2.69$ for the iResNet model.  

For $f_2$, the exponent of exponential convergence is increased by a factor of 1.53, and
1.64, upon employing a linear and iResNet transformation. The use of an
optimized power-law transformation improves convergence by increasing the
constant $\nu$.
In contrast, the use
of the power-law transformation with the asymptotically motivated parameters
renders the convergence super geometric, \ie, $\alpha
> 1$. To understand these results we look at the asymptotics of the
optimal map.

As demonstrated in \autoref{thm:adapt_to_decay}, the optimal map $g$ should be such that 
$$
\exp(-g(x)^2) g'(x) \approx \exp(-x^4/2 ). 
$$
Taking the logarithm of both sides yields
$$
-g(x)^2 +  \log(g'(x)) \approx -\frac{x
^4}{2}.
$$
As \( |x| \to \infty \), the term \( \log(g'(x)) \) is negligible compared to the remaining terms. 
This suggests that \( g(x) \) grows quadratically, i.e., \( g(x) \sim x^2 \), and consequently \( |g'(x)| \sim |x| \).
This behavior is consistent with the optimized power-law model, for which \( \beta \approx -0.3 \).

Motivated by this observation, we considered parameter choices for \( \epsilon \) and \( \beta \) in the transformation \eqref{eq:example_g} based on asymptotic considerations rather than loss-based optimization. 
We found that the choice \( \epsilon = 8 \) and \( \beta = 0.45 \) yields approximations with a significantly improved convergence rate. 
This model was used in the numerical experiments and is referred to as the
power-law model with asymptotically selected parameters (asymp) in \autoref{fig:2}
and \autoref{tab:2}.

Notably, the loss associated with these parameters $\approx 0.98$ is higher than
that obtained \emph{via} optimization $\approx 0.52$. This suggests that the loss functional
may not fully capture the asymptotic features governing convergence.

For $f_1$, the asymptotic analysis of the optimal map suggests that $g(x) \sim
\sqrt{8 \ \log(|x|)}$ as $|x| \to \infty$, which no power law $|x|^{1-2\beta}$
can satisfy for any value $\beta$.

Due to the Lipschitz contraint of the iResNet architecture, the coordinate
transformation modelled by the iResNet is such that $g_r'(x) \leq (1+L)^P$ where
$L$ is the Lipschitz constant of the residual blocks and $P$ is the number of
blocks. Therefore, an iResNet-based model cannot satisfy a growth condition
of the form $g(x) \sim x^2$ as $|x| \to \infty$. This provides a possible explanation for the suboptimal convergence of approximations employing an iResNet-based coordinate transformation for $f_2$ compared to the power-law model.  
\subsection{Application to time-independent Schrödinger equations}
\label{subsection:Morse}
We consider the task of approximating the lowest 23 eigenvalues (also termed energies) of the Morse
operator given by \eqref{eq:Morse}.
  We set the Morse potential parameters $D_e$, $a_M$, and $\mu$ respectively to $42301$ \invcm,
$2.1440 \textup{~\AA}^{-1}$, and $1.5743 \times 10^{-27}$ kg, whereby representing a typical OH stretch in
molecules. 

Approximations to the eigenfunctions $(\Psi_m)_{m=0}^{22}$ of \eqref{eq:Morse} were constructed in the linear span
of Hermite functions, \ie,
\begin{align}
\label{eq:lin_Ansatz}
  \Psi_m &\approx \tilde{\Psi}_m \nonumber\\
&= \sum_{n<N} c_{n,m} h_n, \quad m=0, \dots, 22. 
\end{align} 
To derive a weak formulation of the infinite-dimensional eigenvalue
problem, we chose both the test and trial functions to be equal to the ansatz
functions given in \eqref{eq:lin_Ansatz}. This leads to a finite-dimensional
eigenvalue problem, from which the expansion coefficients \(\{c_{n,m}\}_{n=0,\,
m=0}^{N-1,\,22}\) can be computed.

Similarly, we constructed approximations in the linear span of Hermite functions
expressed in adaptive coordinates, \ie, 
\begin{align}
  \label{eq:nlin_Ansatz}
    \Psi_m &\approx \tilde{\Psi}_m \nonumber\\
  &= \sum_{n<N} c_{n,m} h_n \circ g \sqrt{g'}, \quad m=0, \dots, 22, 
  \end{align}
  where $g$ is a parametrized invertible function. We derived a weak
  Galerkin formulation by setting the trial and test functions to be equal to
  \eqref{eq:nlin_Ansatz}. We modelled the map $g$ \emph{via} a linear model $g_l$
  consisting of a scaling and a shifting parameter, and an iResNet $g=g_\mathrm{r}$,
  composed of 5 blocks, see \autoref{subsec:comp_details} for more information
  on the computational details. The parameters $\theta$
  of $g_l$ and $g_r$ and \(C = \{c_{n,m}\}_{n=0,\,
m=0}^{N-1,\,22}\) were then optimized for $N=23$ by minimizing the loss function
\begin{equation}
  \mathcal{L}(\theta,C) := \sum_{m<23} \langle \tilde{\Psi}_m, H \tilde{\Psi}_m \rangle_{\mathbb{R}}, 
\end{equation}
which represents the Rayleigh quotient of the approximations. The optimization
was performed using the Adam optimizer. 

The optimized parameters were subsequently used to derive approximations for
larger values of $N$. We note that the procedure of learning the parameters
$\theta$ for a certain truncation parameter $N$ and transferring the learned
parameters to other values of $N$ has a variety of computational benefits~\cite{Vogt:JCP163:15,Saleh:JCTC21:5221}.

The relative error $\frac{|\tilde{E}_m - E^{\text{ref}}_m|}{|E^{\text{ref}}_m|}$ of the approximated energies $\tilde{E}_m$ with respect to the reference energies $E^{\text{ref}}_m$ computed using the analytical formula for the Morse eigenvalues is plotted in \autoref{fig:1} as a function of $N$ for the different bases.
 Across all values of the
truncation parameter \(N\), learning coordinates consistently yields more
accurate results, with significantly improved convergence observed for all
energy levels. Moreover, modelling the coordinate transformation \emph{via} the
nonlinear iResNet further enhances the accuracy. 

To gain a concrete insight into the improvement of convergence enabled by the
use of adaptive coordinate transformations, we estimated the convergence of the
approximations as follows. We assumed that the approximation errors decay as
$\mathcal{O}(N^{-l})$ for some $l > 0$ and estimated $l$ by performing linear
regression on the logarithm of the approximation errors as a function of
$\log(N)$ for each energy level and each basis. On average across the 23 energy
levels, the use of a linear coordinate transformation improves the decay rate of
Hermite approximations by a factor of approximately $\approx 4$. The use of an
iResNet-based coordinate transformation improves the convergence rate by a factor of
approximately $\approx 9$. 
The convergence order estimates for each eigenvalue are provided in
\autoref{tab:convergence_exponents} in the appendix. 

The optimization of coordinate transformations also increases the accuracy of
approximations for small values of $N$. For instance, for $N=23$, the use of a
linear coordinate transformation and an iResNet-based transformation improves
the accuracy of the 23rd eigenvalue by a factor of approximately $6$ and $2
\times 10^3$, respectively,
compared to standard Hermite approximations.

As discussed in \autoref{sec:convergence_iHermite}, the results presented in
\autoref{fig:1} can be explained by the decay behavior of the eigenfunctions of
\eqref{eq:Morse}. Standard Hermite approximations exhibit slow convergence due
to a mismatch between the asymptotic decay of these eigenfunctions and the
Gaussian decay \(\exp(-x^2)\) inherent to the Hermite basis. Our findings
indicate that the coordinate transformations \(g_\text{r}, g_\text{l}\), which define the
adaptive Hermite
bases, enhance convergence by modifying the decay properties of the target
functions. In particular, the transformed eigenfunctions \(W_{g_\text{r}^{-1}}
\Psi_m\) and \(W_{g_\text{l}^{-1}} \Psi_m\) 
display Gaussian-like decay.

To test this hypothesis we first recall
  that the Morse eigenfunctions are given by 
  $$
  \Psi_m = N_m z^{(\lambda-m-1/2)} \exp(-z/2) L_m^{(2\lambda -2m -1)(z)}(z),
  $$
  where $\lambda=\frac{\sqrt{2\mu D_e}}{a_M \hbar}$, $z=2 \ \lambda \exp(-x-x_e)$ and $L_m^\alpha$ is the associated Laguerre
  polynomial. We consider the case \( m = 0 \) and analyze the decay behavior of the eigenfunction \( \Psi_0 \) and its transformed counterpart \( W_{g^{-1}} \Psi_0 \) by examining the moments

\begin{equation}
  \label{eq:decay_Morse}
I^2(s) = \int_\mathbb{R} |x|^s \left(W_{g^{-1}} \Psi_0\right)^2 \, dx,
\end{equation}

for \( g = \mathrm{id} \) (corresponding to the standard Hermite basis), for $g=g_l$ and
for \( g = g_r \).
These are then compared against the same quantity computed for a Gaussian
function

\begin{equation}
\label{eq:decay_G}
  I^2(s) = \int_\mathbb{R} |x|^s \exp(-x^2) \, dx.
\end{equation}
    \autoref{fig:6} displays the computed moments as a function of \(s\). The
    results indicate that \(\Psi_0\) exhibits super-Gaussian decay,
    while \(W_{g_l^{-1}} \Psi_0\) and \(W_{g_r^{-1}} \Psi_0\) exhibit a slower
    decay. Interestingly, \(W_{g_r^{-1}} \Psi_0\) aligns more
    closely with the Gaussian behavior than \(W_{g_l^{-1}} \Psi_0\), which is consistent with the superior performance of the iResNet-based Hermite basis compared to the linear-based basis.
    These findings support our hypothesis regarding the role of coordinate transformations in adapting the decay profile.

\begin{figure}[h]
  \centering
\includegraphics[width=0.45\textwidth]{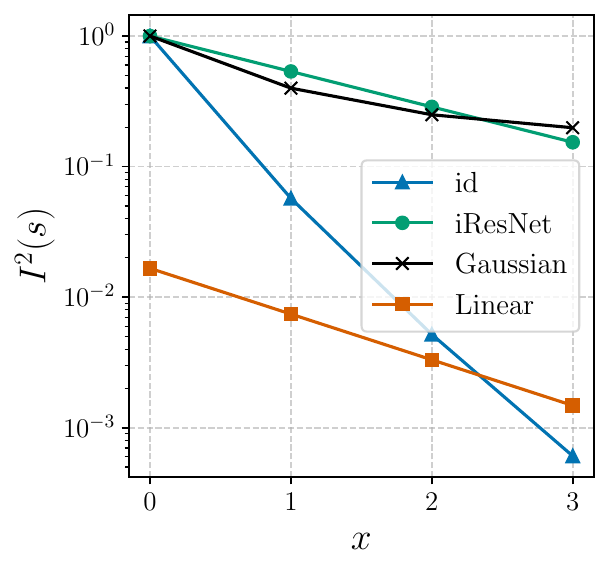}
\caption{Shown are the quantity   \autoref{eq:decay_Morse} as a function of $s$
for $g=\text{id}$ (blue triangles), $g=g_r$ (green
circles), and $g=g_l$ (red squares), and the quantity \autoref{eq:decay_G} (black
crosses).} 
\label{fig:6}
\end{figure}

\section{Conclusion}
Normalizing flows are emerging as a powerful tool for modeling coordinate
transformations. Such transformations were shown to substantially improve the
convergence of Hermite expansions in solving various Schrödinger equations. A
first step towards understanding the underlying mechanisms of this improvement
is to analyze Hermite expansions in adaptive coordinates. 

We established a general principle that allowed us to transfer convergence
theorems of Hermite expansions to approximations by orthonormal bases of the forms
\eqref{eq:flows_ansatz}, see
\autoref{thm:transfer}. This principle
shifts the analysis of approximations in adaptive coordinates to the analysis of
regularity of the pull-back functions $W_{g^{-1}} f$ in the original coordinates.

We then used
this principle to prove the existence of adaptive Hermite bases that yield
improved convergence rates for approximating smooth functions that exhibit an exponential decay along the real axis.
 These results were based on pushing forward the density of the target
function to a Gaussian density. 

Finally, we performed numerical experiments to demonstrate the practical
benefits of using normalizing flows to induce adaptive Hermite bases. We showed
that encoding the insight gained from \autoref{thm:adapt_to_decay} into the loss
function used to optimize the parameters of the coordinate transformations can
significantly improve the convergence rates of Hermite approximations. In particular, we found that the use of a power-law coordinate transformation with parameters selected based on asymptotic considerations can yield super-geometric convergence rates for approximating functions with super-Gaussian decay.

We also optimized coordinate transformations to learn the eigenfunctions of the Morse oscillator. We showed that an iResNet model, designed to
minimize the Rayleigh quotient of the approximations, effectively learns
coordinate transformations that map the eigenfunctions of the Morse oscillator to
functions with Gaussian-like decay, thereby significantly improving the
convergence of Hermite approximations.

In future works, it would be interesting to extend the analysis of
approximations in adaptive coordinates to the multivariate setting, to other
types of orthogonal expansions, and to other functional spaces, such as the
space of continuous functions~\cite{Fernandez:arXiv2512:13740}.

\section*{Code Availability}
The codes we developed to perform the numerical experiments are available at \url{https://github.com/robochimps/num_analysis_flows}.
\printbibliography%
\listofnotes%
\appendix
\section{Appendix}
\subsection{Equivalence principle for Riesz bases}
\label{app:error_comp}

Analogously to \autoref{thm:transfer} we establish a principle that
allows us to transfer convergence results of Hermite expansions to
approximations by Riesz bases of the form \eqref{eq:induced_Riesz}. 

We begin by establishing basic properties of the composition operator $C_g$.
  \begin{theorem}[Basic properties of $C_g$]
    \label{thm:props_Cg}
  Let $g: \Omega \to \mathbb{R}$ be a differentiable map that satisfy the
  hypothesis of \autoref{thm:riesz}. Then, $C_g$
  is invertible and its inverse $C_g^{-1}$ is a composition operator induced
  by the inverse map $g^{-1}$, \ie, $C_g^{-1} = C_{g^{-1}}$.  
\end{theorem}
\begin{proof}
The result then follows from the fact that $C_g$ is a composition operator and that the inverse of a composition operator is also a composition operator induced by the inverse map; see, e.g., \textcite[Corollary 2.2.3]{Singh:CompOp1993}.
\end{proof}

We consider only Riesz bases of the form
  \eqref{eq:induced_Riesz} induced by bi-Lipschitz maps \(g\), \ie,
  maps for which \(g'\) is bounded away from zero and infinity almost
  everywhere. We denote the lower and upper Lipschitz constants by $k, K$, respectively.

We define $P_{N,C}$, the projection
operator associated with the Riesz basis, as
\begin{equation*}
  P_{N, C}: L^2(\Omega) \to \operatorname{span}(C_g h_0, \dots, C_gh_{N-1}), \quad  P_{N, C}: f \mapsto \sum_{n=0}^{N-1} \langle f, (C_g h_n) \ g' \rangle_\Omega \ C_g h_n.
\end{equation*}

\begin{theorem}[Equivalence principle for Riesz bases]
\label{thm:transfer_Riesz}
For any $f \in L^2(\Omega)$, 
\[
\|f-P_{N,C}f\|_{L^2(\Omega)}
\leq
k^{-1/2} \ \|C_{g^{-1}}f - P_N(C_{g^{-1}}f)\|_{L^2(\mathbb{R})}.
\]
\end{theorem}
\begin{proof}
  Since $C_g$ is invertible with inverse $C_{g^{-1}}$,
\[
f - P_{N,C} f
=
C_g\!\left(
C_{g^{-1}} f
-
\sum_{n=0}^{N-1}
\langle f, (C_g h_n)\, g' \rangle_\Omega \, h_n
\right).
\]
The coefficients satisfy
\[
\langle f, (C_g h_n)\, g' \rangle_\Omega
=
\langle C_{g^{-1}} f, h_n \rangle_{\mathbb{R}}.
\]
Setting $\tilde f := C_{g^{-1}} f$, we obtain
\[
f - P_{N,C} f
=
C_g\!\left(
\tilde f
-
\sum_{n=0}^{N-1}
\langle \tilde f, h_n \rangle_{\mathbb{R}} \, h_n
\right).
\]
The proof is completed by noting that $\|C_g\| = k^{-1/2}$.
\end{proof}

\begin{remark}
Note that one can exchange the role of the Riesz basis \eqref{eq:induced_Riesz}
and its dual, \ie, one can consider approximations of the form 
$$
\tilde{P}_{N,C} f := \sum_{n=0}^{N-1} \langle f, C_g h_n \rangle_\Omega \ (C_g h_n) \ g'.  
$$
The equivalence principle in such a case takes the form
\begin{equation}
\|f - \tilde{P}_{N,C} f\|_{L^2(\Omega)} \leq  \|C_{g^{-1}} (f/g') - P_N(C_{g^{-1}}(f/g'))\|_{L^2(\mathbb{R})}.
	\end{equation}
\end{remark}
We note that establishing a corollary similar to \autoref{cor:adapt_to_decay} is
not straightforward in the case of Riesz bases. The
difficulty arises from the fact that the coordinate transformation $g$
constructed in the proof of \autoref{thm:adapt_to_decay} is not guaranteed to be bi-Lipschitz, and thus does not induce a Riesz basis.

\subsection{Computational details}
\label{subsec:comp_details}
In all experiments pertaining to interpolation, we used an iResNet architecture composed of 10 blocks, each
consisting of a fully connected feedforward network with 2 hidden layers of 8
neurons and one output neuron. We employed the lipswish activation function.

To solve the optimization problem \eqref{Eq:opt_problem} we used the same 10000 uniformly spaced sample points used to construct the
target functions. We then applied the Adam optimizer with a learning rate of
$10^{-3}$ for 20000 iterations for $i=1$ for the iResNet-based model. For $i=2$,
we used a learning rate of $3 \times 10^{-2}$ for 40000 iterations for the
iResNet-based model.
For the linear
model, we used the Adam optimizer with a learning rate of $10^{-2}$ for 1000
iterations for $i=1$, and a learning rate of $10^{-3}$ for 10000 iterations for
$i=2$. For the power-law model, we used the Adam optimizer with a learning rate
of $10^{-3}$ for 30000 iterations for $i=1$, and a learning rate of $3 \times
10^{-6}$ for 10000 iterations for $i=2$. For all optimizations, we employed the cosine annealing learning rate schedule for both target functions.

For optimizing the iResNet model for the Schrödinger equation, we used the same
architecture described in \textcite{Vogt:JCP163:15} and the Adam optimizer with
a learning rate of $10^{-3}$ for 20000 iterations. For the training we used Hermite
quadrature with $100$ quadrature points. For optimizing the linear model, we used the Adam optimizer with a learning rate of $10^{-3}$ for 10000 iterations.

\begin{table}[H]
\centering
\begin{tabular}{c|ccc}
\hline
Eigenvalue & Hermite & Linear & iResNet \\
\hline
0 & 1.506 & \textbf{21.575} & 6.434 \\
1 & 13.099 & \textbf{20.398} & 12.932 \\
2 & 16.612 & 18.845 & \textbf{19.632} \\
3 & 14.676 & 16.970 & \textbf{25.519} \\
4 & 12.325 & 17.015 & \textbf{29.531} \\
5 & 9.632 & 17.776 & \textbf{31.434} \\
6 & 6.920 & 17.050 & \textbf{31.166} \\
7 & 4.794 & 15.533 & \textbf{29.879} \\
8 & 3.441 & 14.087 & \textbf{28.437} \\
9 & 2.640 & 12.631 & \textbf{27.279} \\
10 & 2.150 & 11.504 & \textbf{26.564} \\
11 & 1.831 & 10.936 & \textbf{25.923} \\
12 & 1.612 & 10.760 & \textbf{25.038} \\
13 & 1.455 & 10.491 & \textbf{24.012} \\
14 & 1.337 & 10.232 & \textbf{22.833} \\
15 & 1.251 & 8.379 & \textbf{21.260} \\
16 & 1.184 & 6.023 & \textbf{19.232} \\
17 & 1.129 & 4.178 & \textbf{16.491} \\
18 & 1.113 & 3.900 & \textbf{13.035} \\
19 & 1.065 & 2.915 & \textbf{10.570} \\
20 & 1.029 & 2.962 & \textbf{10.801} \\
21 & 1.275 & 2.573 & \textbf{11.078} \\
22 & 2.014 & 2.468 & \textbf{2.529} \\
\hline
\end{tabular}
\caption{Convergence exponents for the first 23 eigenvalues of the Morse oscillator \eqref{eq:Morse} using Hermite functions, Hermite functions with a linear coordinate transformation, and Hermite functions with an iResNet-based transformation.}
\label{tab:convergence_exponents}
\end{table}

\end{document}

%% Local Variables:
%% coding: utf-8
%% mode: LaTeX
%% mode: auto-fill
%% mode: flyspell
%% fill-column: 100
%% ispell-dictionary: "american"
%% reftex-cite-format: default
%% TeX-auto-save: t
%% TeX-close-quote: "''"
%% TeX-open-quote: "``"
%% TeX-parse-self: t
%% truncate-lines: t
%% End: